\documentclass[11pt]{article}

\usepackage[top=1in, bottom=1.25in, left=0.75in, right=0.75in]{geometry}
\usepackage{hyperref}
\usepackage{amssymb}
\usepackage[linesnumbered,lined,ruled]{algorithm2e}
\usepackage{amsmath,amsthm,amsfonts,dsfont,mathtools,bm}
\usepackage{url}
\usepackage{rotating}
\usepackage{multirow}
\usepackage{color}
\usepackage{xcolor}
\usepackage{enumitem}
\DeclareSymbolFont{rsfs}{U}{rsfs}{m}{n}
\DeclareSymbolFontAlphabet{\mathscrsfs}{rsfs}

\allowdisplaybreaks
\newtheorem{theorem}{Theorem}
\newtheorem{definition}[theorem]{Definition}

\newtheorem{lemma}[theorem]{Lemma}

\newtheorem{corollary}[theorem]{Corollary}

\newtheorem{remark}{Remark}

\renewcommand{\P}{\operatorname{\mathbb{P}}}

\newcommand{\E}{\operatorname{\mathbb{E}}}

\newcommand{\R}{\mathbb{R}}

\newcommand{\pr}{\mathbb{P}}

\newcommand{\indi}{\mathds{1}}

\newcommand{\rmd}{\mathrm{d}}

\renewcommand{\hat}{\widehat}
\def\Law{{\rm Law}}

\newcommand{\bG}{G}
\newcommand{\bg}{g}

\newcommand{\cF}{{\mathcal F}}

\newcommand{\Unif}{{\sf Unif}}

\newcommand{\salg}{\mbox{\rm\tiny alg}}

\newcommand{\eps}{\varepsilon}

\newcommand{\pl}{\mbox{\rm\tiny pl}}
\newcommand{\rd}{\mbox{\rm\tiny rd}}

\def\bx{x}

\def\by{y}

\def\b0{{\boldsymbol{0}}}

\def\ALG{{\sf ALG}}

\def\sSAT{\mbox{\tiny \sf SAT}}
\def\sALG{\mbox{\tiny \sf ALG}}

\def\Law{{\rm Law}}

\def\<{{\langle}}
\def\>{{\rangle}}

\def\cC{\mathcal{C}}

\numberwithin{equation}{section}
\numberwithin{theorem}{section}

\begin{document}

\title{Hardness of sampling solutions from the Symmetric Binary Perceptron}

\author{Ahmed El Alaoui\thanks{Department of Statistics and Data Science, Cornell University. Email: elalaoui@cornell.edu}, \;\; David Gamarnik \thanks{Sloan School of Management, Massachusetts Institute of Technology. Email: gamarnik@mit.edu}}

\date{}
\maketitle

\begin{abstract}
We show that two related classes of algorithms, stable algorithms and  Boolean circuits with bounded depth, cannot produce an approximate sample from the uniform measure over the set of solutions to the symmetric binary perceptron model at any constraint-to-variable density. This result is in contrast to the question of finding \emph{a} solution to the same problem, where efficient (and stable) algorithms are known to  succeed at sufficiently low density. This result suggests that the solutions found efficiently---whenever this task is possible---must be highly atypical, and therefore provides an example of a problem where search is efficiently possible but approximate sampling from the set of solutions is not, at least within these two classes of algorithms.
\end{abstract}


\section{Introduction}\label{sec:1}
We consider the problem of approximately sampling from the uniform distribution over the set of solutions of the  Symmetric Binary Perceptron (SBP) model. The model is described
as follows.
Let $\bG = (\bg_a)_{a=1}^m$ be $m$ independent random vectors in $\R^n$ defined over a probability space $(\Omega,\cF,\P)$, and let $\alpha, \kappa > 0$. We let $m = \lfloor \alpha n \rfloor$ and denote by $S(\bG, \kappa)$ the set of binary solutions $\bx \in \{-1,+1\}^n$ to the system of inequalities
\begin{equation}
\big| \langle \bg_a , \bx \rangle \big| \le \kappa \sqrt{n} ~~~ \mbox{for all}~ 1 \le a \le m\, .
\end{equation}
We consider two distributions over the input $\bG$: the Gaussian case 
where each entry is distributed independently according to $N(0,1)$, 
and the Rademacher case where each entry is an independent symmetric 
random sign. In both cases (in fact under
more general distributional assumptions) a sharp and simple characterization of the 
probability that $S(\bG, \kappa)\ne\emptyset$ 
was recently derived in \cite{abbe2022proof,perkins2021frozen}. 
For every $\kappa>0$,
this set is non-empty when $\alpha<\alpha_{\sSAT}(\kappa) := -(\log 2)/\log \P(|Z|\le \kappa)$ and empty when $\alpha>\alpha_{\sSAT}(\kappa)$, both with high probability (whp)
as $n\to\infty$. Here $Z$ denotes a standard Gaussian random variable.

It was further established in~\cite{abbe2022proof,perkins2021frozen} that the set of solutions
exhibits the following interesting geometric property: For every $\kappa>0$ at any positive constraint-to-variable density $\alpha>0$
the space of solutions consists mostly (i.e., with exponentially small size exception) 
of isolated points separated by linear 
distance. 
Specifically, for every $\alpha>0$, whp as $n\to\infty$
there exists a subset $S_0(\bG, \kappa)\subset S(\bG, \kappa)$, such that 
(a) for every $\bx\ne\by\in S_0(\bG, \kappa)$, $d_{H}(\bx,\by)=\Omega(n)$ and
(b) $|S_0^c(\bG, \kappa)|/|S(\bG, \kappa)|\le\exp(-\Omega(n))$. Here $A^c$ denotes
the set-theoretic complement of $A$.

Based on similarities to other randomly generated constraint satisfaction 
problems, such as the random K-SAT problem~\cite{AchlioptasCojaOghlanRicciTersenghi}, this  suggests
an algorithmic difficulty of finding at least one satisfying solution. 
Surprisingly, this is not the case and for every $\kappa>0$ at small enough density $\alpha < \alpha_{\sALG}(\kappa) <\alpha_{\sSAT}(\kappa)$, polynomial time algorithms for finding
 solutions whp as $n\to\infty$ exist~\cite{bansal2020line,abbe2022binary}. 
This apparent discrepancy was explained in~\cite{gamarnik2022algorithms} where it was shown that a provable barrier to algorithms finding 
a solution appears in the form of an Overlap Gap Property (OGP), which occurs 
near $\alpha_{\sALG}(\kappa)$ when $\kappa$ is small. 

In this paper we consider the question of whether it is possible to efficiently generate \emph{typical} solutions from $S(\bG,\kappa)$. 
In particular whether it is possible to sample
from it nearly uniformly at random. 
We provide a strong evidence that the answer
is `no' for all positive densities $\alpha>0$. Specifically, for two
classes of algorithms we describe below, we show that the distribution produced
by these algorithms deviates from the uniform distribution on $S(\bG, \kappa)$
by a significant margin, and thus such algorithms fail
to sample nearly uniformly from the set of solutions. 
This is the main result of our paper. One interpretation of this result is that the solutions produced by efficient algorithms as previously mentioned must be highly atypical.        

The two algorithmic classes we consider are (a) stable algorithms (appropriately
defined) and (b) a related class of bounded depth  Boolean circuits. Stable algorithms are loosely speaking
algorithms which have low noise sensitivity. Specifically, small changes in the input
matrix $\bG$ results in a small change in the output of the algorithm. This 
is quantified more precisely in Section~\ref{section:Stable-circuits}.
Many algorithms which find solutions in random constraint satisfaction
problems are known to be stable~\cite{gamarnik2021survey} and some
are known to be effective for approximate sampling, as was shown recently
in~\cite{alaoui2022sampling,alaoui2023samplingpspin,huang2024sampling}
in the context of mean-field spin glass models.

As an illustration of stability in the context of sampling, we mention $(i)$ \emph{sequential sampling} where $x_1$ is sampled from its marginal, then for every $i \ge 2$, $x_i$ is sampled from (an approximation of) its conditional distribution given $x_1,\cdots,x_{i-1}$, and $(ii)$ the class of \emph{score-based diffusion algorithms}~\cite{song2019generative} or algorithms based on discretizations of the \emph{stochastic localization process}~\cite{chen2022localization} which are in a certain sense continuous-time versions of the sequential sampling procedure; see~\cite{chen2022localization,montanari2023sampling} for an exposition of a common framework, as well as other sampling schemes from the same family. Such algorithms can be shown to be stable as long as the subroutine estimating the sequence of conditional distributions in the case $(i)$ or for approximating the `denoising' or `score' function in the case $(ii)$ is suitably Lipschitz as a function of its input.           

As we mentioned earlier, when one is interested in search rather than sampling, i.e., one is only interested in finding a solution, and not necessarily a  typical one, then OGP can be used to argue impossibility of the search task by stable algorithms. 
Another class of algorithms ruled out by OGP in the context of the search problem is our (b) example, namely bounded
depth polynomial size Boolean circuits. This class of algorithms exhibits a form of `average' stability, in contrast to the usual definition of stability used to argue against search algorithms, which is uniform in the input.  
Technically, this comes in the form of the celebrated Linial-Mansour-Nisan (LMN) Theorem~\cite{o2014analysis} on the concentration of the Fourier spectrum of a Boolean circuit on its low degree projection.
This average notion turns out to be sufficient for our sampling problem, and we use it to rule out two special classes of Boolean circuits:
polynomial size circuits 
with depth $c\log n/\log\log n$
for small enough constant $c$,
which depends on $\kappa$ and $\alpha$, and circuits with bounded
depth and size
at most $\exp(O(n^c))$,
where $c$ again depends on $\kappa$ and $\alpha$.

Our proof approach is similar to the one of \cite{alaoui2022sampling}
and is based on establishing that the model exhibits the so-called
\emph{Transport Disorder Chaos} (TDC). Roughly speaking this property says 
that the measure is unstable to small perturbation to the input $\bG$ when
a transport (Wasserstein-2) distance is used between measures.
The proof of the TDC property is somewhat technical and uses an associated planted model; an instrumental tool previously used in deriving
the satisfiability threshold $\alpha_{\sSAT}(\kappa)$ 
in~\cite{abbe2022proof,perkins2021frozen}. TDC implies failure
of stable algorithms via a simple triangle inequality. On the other hand we show that Boolean circuits are stable via the LMN Theorem. 
Part of this argument is similar to the one found
in~\cite{gamarnik2020optimization} where it was used to show the limitation of 
circuits for solving  search problems.


While our result  rules out only certain classes of 
sampling algorithms, we believe that the hardness of this problem is genuine, and in 
fact no polynomial time algorithms for sampling exist for this problem at any
density $\alpha>0$. Needless to say, this is not within a near reach as such a 
result  would imply separation between $P$ and $\#P$ classes of algorithms. 
An interesting extension of this work would be showing a failure of broader
classes of algorithm. We note
that a canonical choice of sampling technique, namely Monte Carlo Markov Chain/Glauber dynamics is not a viable 
option here since the solution space is not connected and as such is not ergodic. 
It seems plausible though to achieve sampling using Simulated Annealing (SA) technique. We believe though that the time for SA to 
reach steady-state is exponentially large 
and leave it as an open question. 

Finally, another very interesting question which remains open is whether it is 
possible to find any solution in the bulk $S_0(\bG, \kappa)$ (the set of isolated solutions) at least 
within the same classes of algorithms (stable or low-depth circuits). We suspect
that the answer is `no', but our current proof technique does not extend to this result.
We leave it as a very interesting open problem.

\section{Instability of the set of solutions}
In this section we establish a certain structural property for this uniform measure called Transport Disorder Chaos in~\cite{alaoui2023shattering}, and then show that this presents a barrier to approximate sampling for the family of stable algorithms and for Boolean circuits of bounded depth and/or bounded size. At a high level, TDC is the property that the measure is unstable to small perturbations of the input data, in the sense that it moves a substantial fraction of its mass to a linear distance away from its initial location. Conversely, stable algorithms are insensitive to such small perturbations to their input, and are thus unable to produce a typical sample from the measure. 
We will define the notion of perturbation differently depending on the distribution of the input:

\begin{enumerate}
\item \label{it:g} {\bf The Gaussian case:} If $\bG$ has independent standard normal entries, we consider a perturbation of $\bG$ defined by the interpolation 
\begin{equation*} 
\bG_t = t \, \bG + \sqrt{1-t^2} \, \bG',~~~ t\in [0,1] \, ,
\end{equation*}
where $\bG'$ is an independent copy of $\bG$.
\item  \label{it:r} {\bf The Rademacher case:} If $\bG$ has independent random sign entries, we instead construct $\bG_t$ by resampling each entry of $\bG$ independently with probability $1-t \in [0,1]$.
\end{enumerate}
In both cases, $\bG$ and $\bG_t$ have marginally the same law and are correlated entrywise with correlation $t$. 
Next, we measure displacement of mass due to perturbation via the Wasserstein-2 distance defined as 
\begin{equation*}  
W_{2}(\mu,\nu)^2 =  \inf_{\pi \in \cC(\mu,\nu)} 
\E_{\pi} \Big[\big\|X - Y\big\|_2^2\Big] \, , 
\end{equation*}   
where $\mu , \nu$ are any two probability 
measures on $\R^n$ with finite second moments, and the infimum is over all couplings $(X,Y) \sim \pi$ with marginals $X \sim \mu$ and $Y \sim \nu$. 
Finally, for a finite set $S$, $\Unif(S)$ denotes the uniform distribution over $S$.

We now state our main structural result that $\Unif\big(S(\bG,\kappa)\big)$ has transport disorder chaos: 
\begin{theorem} \label{thm:w_2-chaos} 
Let $\bG$ be either drawn from the Gaussian or Rademacher distribution and let $\bG_t$ be defined accordingly. Assume that $\kappa>0$ and $\alpha < \alpha_{\sSAT}(\kappa)$.    
Let $\mu_{\bG} = \Unif\big(S(\bG,\kappa)\big)$ if $S(\bG,\kappa) \neq \emptyset$, and define $\mu_{\bG}$ arbitrarily otherwise.  
Then there exists $\delta, C>0$ and $t_0 \in [0,1)$, all depending on $\alpha, \kappa$ such that for any sequence $t_n\in [t_0,1)$ satisfying
\begin{itemize}
\item $n \sqrt{1-t_n} \ge C \log n$  in the Gaussian disorder case, or
\item $\sqrt{n(1-t_n)} \ge C \log n$ in the Rademacher disorder case,
\end{itemize} 
we have
\begin{equation}
\lim_{n \to \infty} \, \P\Big( W_{2}\big(\mu_{\bG} , \mu_{\bG_{t_n}}\big) \ge  \sqrt{\delta n}   \Big) = 1 \, .
\end{equation}
\end{theorem}

\begin{remark}
The assumptions on $t_n$ essentially require that $t_n$ approaches $1$ not faster than order $(\log n/n)^2$ for the Gaussian disorder case,
and $(\log n)^2/n$ for the Rademacher disorder case. The weaker rate we obtain in the latter is due to the use of a central limit theorem where additional errors are incurred relative to the Gaussian case. We also note that the logarithmic term can be dispensed with, i.e., a sufficient condition should be $n\sqrt{1-t_n} \to \infty$, at least in the Gaussian case in the presence of a contiguity statement between two probability models as we describe below.
\end{remark}

\paragraph{The planted model} As a preliminary to the proof we introduce a \emph{planted} distribution on pairs $(\bG,\bG_t)$ and state a few useful results from~\cite{perkins2021frozen,abbe2022proof}. 
We first let $\P_{\rd}$ denote our \emph{null} or \emph{random} distribution on pairs $(\bG,\bG_t)$ as in the description in \ref{it:g} in the Gaussian case, or the description in \ref{it:r} in the Rademacher case. Next we let $\P_{\pl}$ be the joint distribution on pairs $(\bG,\bG_t)$ defined by the density 
\begin{equation}\label{eq:lr}
\frac{\rmd \P_{\pl}}{\rmd \P_{\rd}} (\bG , \bG_t) = \frac{\big|S(\bG,\kappa)\big|}{\E_{\rd}\big|S(\bG,\kappa)\big|} \, .
\end{equation}
Note that the above density does not depend on its second argument $\bG_t$. An equivalent way of describing the planted distribution is to first generate a \emph{planted configuration} $\bx$ uniformly from $\{-1,+1\}^n$, then generate $\bG$ from the Gaussian or Rademacher distribution conditionally on the event $\bx \in S(\bG,\kappa)$, and then generate $\bG_t$ according to either description \ref{it:g} or \ref{it:r}.  

The first result is about the magnitude of the likelihood ratio in Eq.~\eqref{eq:lr} which will allows to transfer statements established under the planted model to equivalent statements under the null model. 

\begin{lemma}[\cite{perkins2021frozen,abbe2022proof}]\label{lem:lr}
Let $L = \rmd \P_{\pl} / \rmd \P_{\rd}$. 
For all $\alpha < \alpha_{\sSAT}(\kappa)$ there exists a constant $C>0$ such that    
\[ \lim_{n \to \infty} \P_{\rd}\Big(|\log L| \le C\log n\Big) = 1\, ,\]
in both the Rademacher and the Gaussian cases. 
\end{lemma}

The above result was proved by Perkins and Xu~\cite{perkins2021frozen} under a numerical assumption, first identified in~\cite{aubin2019storage}, about the possible maxima of a certain univariate function. Subsequent work of Abb\'e, Li and Sly~\cite{abbe2022proof} verified this condition and proved the stronger result that $L$ has a log-normal limit (of constant order) when the disorder is Rademacher. Subsequent work of Sah and Sawhney~\cite{sah2023distribution} extended the latter result to the Gaussian case very close to the satisfiability threshold, i.e., when $m = \alpha_{\sSAT}(\kappa)n - \eta \log n$, $\eta>0$. Log-normality of $L$ implies that the planted and null distributions are mutually \emph{contiguous}, meaning that any (sequence of) rare event(s) under one distribution is also rare under the other~\cite[Lemma 6.4]{van1998asymptotic}.      
For our purposes the weak statement in the above lemma is enough, as we will only need to transfer statements about events which are super-polynomially rare as we clarify in the next lemma.         
 
\begin{lemma}\label{lem:contig}
Suppose $(E_n)$ a sequence of events in $\mathcal{F}$ such that $\P_{\pl}(E_n) = o(n^{-C})$ where $C$ is the constant appearing in Lemma~\ref{lem:lr}. Then  $\P_{\rd}(E_n) =o_n(1)$.  
\end{lemma}
\begin{proof}
We have
\begin{align*}
\P_{\rd}(E_n) &= \E_{\pl}\big[\indi\{E_n, L>n^{-C}\} L^{-1}\big] 
+ \E_{\rd}\big[\indi\{E_n, L<n^{-C}\}\big]\\
&\le n^C \P_{\pl}(E_n) + \P_{\rd}(L<n^{-C})\, . 
\end{align*}
The above tends to zero since $n^{C}\P_{\pl}(E_n) \to 0$ and as a consequence of Lemma~\ref{lem:lr} for the second term. 
\end{proof}
 
Now we are ready to prove the main theorem. 
 \subsection{Proof of Theorem~\ref{thm:w_2-chaos}}
The main idea is that conditional on $S(\bG,\kappa)$ not being empty and for $t$ not too close to 1, with high probability there is no solution $\by \in S(\bG_t,\kappa)$ to the interpolated problem within a Hamming distance $n \delta$ from a \emph{randomly chosen} solution $\bx \in S(\bG,\kappa)$ to the original problem, for some $\delta = \delta(\alpha,\kappa)>0$ not depending on $t$. 
For $t<1$ and $\delta>0$ and two arrays $\bG, \bG_t \in \R^{m \times n}$, consider the set 
\begin{equation}
E_{\delta,t} = \Big\{ \bx \in S(\bG, \kappa) \,:\, \forall \, \by \in S(\bG_t,\kappa)\, , ~d_{H}(\bx , \by) > \delta n  \Big\} \, ,
\end{equation}
where $d_{H}$ again denotes the Hamming distance between binary vectors.
We observe that under any coupling $\pi$ of $\mu_{\bG}$ and $\mu_{\bG_t}$, we have
\begin{equation}
\E_{(X,Y)\sim \pi} \Big[\big\|X - Y\big\|_2^2\Big] = 4\E_{X \sim \mu_{\bG}}\Big[\E\big[d_{H}(X,Y) \, | \, X \big]\Big] \ge 4\delta n\, \mu_{\bG}(E_{\delta,t})\, .
\end{equation}
Therefore it suffices to lower-bound $\mu_{\bG}(E_{\delta,t})$. We will in fact show that $\mu_{\bG}(E_{\delta,t}) \to 1$ in probability for some $\delta>0$ not depending on $t$ as long as $n \sqrt{1-t} \ge C \log n$ in the Gaussian case and $\sqrt{n} \sqrt{1-t} \ge C \log n$ in the Rademacher case.
We achieve this via a first moment method in the planted distribution. 
Our goal will be to show that $\E_{\pl}\big[\mu_{\bG}(E_{\delta,t}^c)\big] \to 0$ as $n \to \infty$. 
By using Eq.~\eqref{eq:lr} we have
 \begin{align}\label{eq:mu^c}
 \E_{\pl}\big[\mu_{\bG}(E_{\delta,t}^c)\big]
 &=   \E_{\pl} \Big[ \frac{1}{|S(\bG, \kappa)|} \sum_{\bx \in \{-1,+1\}^n} \indi\Big\{\bx \in S(\bG, \kappa) \,, \exists \, \by \in S(\bG_t,\kappa)\,\, \mbox{s.t.} \,\, d_{H}(\bx , \by) \le \delta n\Big\} \Big]\nonumber\\
&= \frac{1}{\E|S(\bG,\kappa)|} \sum_{\bx \in \{-1,+1\}^n}  \P_{\rd}\Big(\bx \in S(\bG, \kappa) \,,~ \exists \, \by \in S(\bG_t,\kappa)\,\, \mbox{s.t.} \,\,d_{H}(\bx , \by) \le \delta n\Big) \nonumber\\
&= \frac{1}{2^n} \sum_{\bx \in \{-1,+1\}^n}  \P_{\rd}\Big(\exists \, \by \in S(\bG_t,\kappa)\,\, \mbox{s.t.} \,\, d_{H}(\bx , \by) \le \delta n \, \Big|\, \bx \in S(\bG, \kappa)\Big)\, .
 \end{align}    
   
By sign symmetry of the distribution of $(\bG,\bG_t)$, the probability in the above display does not depend on $\bx$, which we may assume fixed in all subsequent calculations. We also omit the subscript `$\mathrm{rd}$'. Let $\by \in \{-1,1\}^n$, $\bg_t$ be distributed as any row of $\bG_t$ and $\bg$ be distributed as a row of $\bG$ (under $\P_{\rd}$). By Markov's inequality and independence of the rows of $\bG$ and $\bG_t$, we obtain from Eq.~\eqref{eq:mu^c}
\begin{align}
\E_{\pl}\big[\mu_{\bG}(E_{\delta,t}^c)\big] &\le \E\Big[\Big|\Big\{\by \in S(\bG_t,\kappa)   \,:\,  d_{H}(\bx,\by) \le  n \delta\Big\}\Big| \, \big|\, \bx \in S(\bG, \kappa) \Big]\nonumber\\ 
&= \sum_{\by:d_{H}(\bx,\by) \le  n \delta} \P\Big(\big| \langle \bg_t , \by \rangle \big| \le \kappa \sqrt{n}  \,\Big|\, \big|\langle \bg , \bx \rangle \big| \le \kappa \sqrt{n} \Big)^m \, .\label{eq:second}
\end{align}
\paragraph{The Gaussian case:} When the disorder is Gaussian  the conditional probability in the above reads
\begin{align}\label{eq:prob}
\P\Big(\big| \langle \bg_t , \by \rangle \big| \le \kappa \sqrt{n} \,   \Big| \,  \big| \langle \bg , \bx \rangle \big| \le \kappa \sqrt{n} \Big) 
&= \P\Big(\big| \langle \bg_{t} , \by \rangle \big| \le \kappa \sqrt{n}  \, , \big| \langle \bg , \bx \rangle \big|  \le \kappa \sqrt{n}  \Big)/ \P\Big(\big| \langle \bg , \bx \rangle \big| \le \kappa \sqrt{n} \Big) \nonumber \\
&= \P\big(| Z_0 |  \le \kappa \, , | Z_r | \le \kappa \big)/ \P\big(| Z_0 | \le \kappa \big) \, ,
\end{align}
where $r = t \, \langle \bx , \by \rangle / n$, and $(Z_0,Z_r)$ is a pair of centered Gaussian random variables with unit variances and correlation $\E[Z_0Z_r]=r$. 
Let us use the notation 
\begin{equation}
p(\kappa) =  \P\big( | Z_0 | \le \kappa\big) \, , ~~~~ \mbox{and}~~~~ q_{\kappa}(r) = \P\big( | Z_0 | \le \kappa \, ,\, |Z_r | \le \kappa \big) \, .\nonumber
\end{equation}
We will need facts about the monotonicity of $q_{\kappa}$ and estimates on its behavior near $r=1$:
\begin{lemma}\label{lem:estimates_q}
The map $r \mapsto q_{\kappa}(r)$ is even and non-decreasing  on $[0,1]$. Furthermore, there exists $r_0 \in (0,1)$ and $c>0$ both depending on $\kappa>0$ such that for all $r \in [r_0,1)$,
\begin{equation} \label{eq:ubq}
q_{k}(r) - p(\kappa) \le -c\sqrt{1-r}\, .
\end{equation}
\end{lemma}
We delay the proof of the above lemma to the end of this section.   
Returning to our main argument, for the case $\by = \bx$ in the sum~\eqref{eq:second}, we have for all $t>0$,
\begin{align}\label{eq:1tolast_g}
\P\Big(\big| \langle \bg_t , \bx \rangle \big| \le \kappa \sqrt{n} \,   \Big| \,  \big| \langle \bg , \bx \rangle \big| \le \kappa \sqrt{n} \Big)^m &=  \Big(q_{\kappa}(t) \big/  p(\kappa)\Big)^{m} \, .\nonumber
\end{align}
Next, to treat  the remaining terms, since $q_{\kappa}$ is even and non-decreasing in $t \in [0,1]$, the probabilities in~\eqref{eq:second} are largest for $t=1$. 
Therefore,
\begin{equation}
\E_{\pl}\big[\mu_{\bG}(E_{\delta,t}^c)\big] \le \Big(q_{\kappa}(t) \big/  p(\kappa)\Big)^{m} + \E\Big[\Big|\Big\{\by \in S(\bG,\kappa) \setminus\{\bx\}  \,:\,  d_{H}(\bx,\by) \le  n \delta\Big\}\Big| \, \big|\, \bx \in S(\bG, \kappa) \Big] \, .\nonumber
\end{equation}

The second term counts the number of solutions near a typical one in the planted model. Perkins and Xu~\cite{perkins2021frozen} showed the following bound in the Gaussian disorder case:     
\begin{lemma}\label{lem:isolatedG}
For all $\alpha< \alpha_{\sSAT}(\kappa)$, there exists $\delta =  \delta(\alpha,\kappa)>0$ and $c = c(\delta)>0$ such that for any fixed $\bx \in \{-1,+1\}^n$,
 \begin{equation}
 \E\Big[\Big|\Big\{\by \in S(\bG,\kappa) \setminus\{\bx\}  \,:\,  d_{H}(\bx,\by) \le  n \delta\Big\}\Big| \, \big|\, \bx \in S(\bG, \kappa) \Big] \le e^{-c \sqrt{n}} \, .\nonumber
 \end{equation}
\end{lemma}

Using Lemma~\ref{lem:isolatedG}, there exists $\delta_0 = \delta_0(\alpha,\kappa)>0$ (not depending on $t$) and $c = c(\delta)>0$ such that
\begin{equation}
\E_{\pl}\big[\mu_{\bG}(E_{\delta_0,t}^c)\big] \le p(t)^{\alpha n} + e^{- c\sqrt{n}} \, ,\nonumber
\end{equation}
 where $p(t) = q_{\kappa}(t) / p(\kappa)$. 
By Lemma~\ref{lem:estimates_q}, $p(t) \le 1- c \sqrt{1-t}$ for $t$ close enough to 1, for some constant $c = c(\kappa)>0$. Therefore, as long as 
\[c\alpha  n \sqrt{1-t} -  C\log n \to +\infty\] 
where $C$ is the constant appearing in Lemma~\ref{lem:contig} we have $n^C p(t)^{\alpha n} \to 0$ and we can apply Lemma~\ref{lem:contig} to deduce that $\E_{\rd} \big[\mu_{\bG}(E_{\delta_0,t}^c)\big] \to 0$, and therefore
\begin{equation}
\lim_{n\to \infty} \,\,\mu_{\bG}(E_{\delta_0,t}) = 1~~~~~~\mbox{in prob.\ under} ~\P_{\rd}\,.
\end{equation} 
This yields desired result. It remains to  prove Lemma~\ref{lem:estimates_q} to conclude the argument in the Gaussian case.

\begin{proof}[Proof of Lemma~\ref{lem:estimates_q}]
Monotonicity follows from the formula    
\[q_{\kappa}'(t) = \frac{1}{\pi} \frac{e^{-\kappa^2/2}}{\sqrt{1-t^2}} \left(e^{-\kappa^2(1-t)/2(1+t)} -e^{-\kappa^2(1+t)/2(1-t)}\right) \, ,\]
which is positive for $t \ge 0$. This is obtained by differentiating $q_{\kappa}(t)$ and using the properties of the Gaussian distribution. The computation is relegated to Appendix~\ref{app:blah}. 
%
Now we show the bound~\eqref{eq:ubq}. Letting $\phi$ denote standard normal density and $\varepsilon>0$ we have
\begin{align}
p(\kappa) - q_{\kappa}(t) = \pr(|Z_t|>\kappa, |Z_0|\le \kappa)
&\ge \pr\left(Z_t>\kappa, Z_0\in \left(\kappa-\varepsilon,\kappa\right)\right) \notag\\
&=\int_{\kappa-\varepsilon}^\kappa
\pr(Z_r\ge \kappa| Z=z)\phi(z)dz. \label{eq:Zt-integrated}
\end{align}
We use the fact that
condition on $Z_0=z$, $Z_t$ is distributed
as normal random variable with mean $tz$ and variance
$1-t^2$. Therefore, letting $X$ denote a standard normal
\begin{align*}
\pr(Z_t\ge \kappa| Z=z)=\pr\left(X\ge \frac{\kappa-tz}{\sqrt{1-t^2}}\right).
\end{align*}
When $z\in (\kappa-\varepsilon,\kappa)$,
\begin{align*}
{\kappa-tz\over \sqrt{1-t^2}}
&\le
\kappa {1-t \over \sqrt{1-t^2}}+{\varepsilon t\over \sqrt{1-t^2}} \\
&=\kappa {\sqrt{1-t} \over \sqrt{1+t}}+{\varepsilon t\over\sqrt{1-t^2}} \\
&\le \kappa+{\varepsilon \over \sqrt{1-t^2}}.
\end{align*}
Now taking $\varepsilon = \kappa \sqrt{1-t^2}$,
the integral in (\ref{eq:Zt-integrated}) is at least
$\pr(X \ge 2\kappa) \pr(\kappa - \varepsilon
<Z<\kappa)\ge C(\kappa) \varepsilon$ if $t$ is close enough to $1$. The proof of Eq.~\eqref{eq:ubq} is complete.
\end{proof}

\paragraph{The Rademacher case:} When the disorder is Rademacher we restart the argument from Eq.~\eqref{eq:second} and follow the same outline by invoking the following two Berry-Esseen bounds: 
\begin{lemma}\label{lem:clt}
Let $X_1,\cdots,X_n$ be i.i.d.\ vectors in $\R^2$ where $X_i = (X_{i,1},X_{i,2})$ is a bivariate vector of two unbiased Rademacher r.v.'s having correlation $t$. Let $\xi \in \{-1,+1\}^n$ be fixed and 
let $S_{n,1} = (1/\sqrt{n}) \sum_{i=1}^n X_{i,1}$, $S_{n,2} = (1/\sqrt{n}) \sum_{i=1}^n \xi_i X_{i,2}$. 
There exists a universal constant $C>0$ such that for all $\kappa>0$ and $t \in [0,1]$ we have 
\begin{align}\label{eq:clt}
|\P\big(S_{n,1} \in [-\kappa,\kappa]\big) - p(\kappa)| \le \frac{C}{\sqrt{n}} \,,~~~\mbox{and}~~~
\P\big(S_n \in [-\kappa,\kappa]^2\big) - q_{\kappa}\Big(\frac{t}{n}\sum_{i=1}^n\xi_i\Big) \le \frac{C\log n}{\sqrt{n}} \, ,
\end{align}
where $S_n = (S_{n,1},S_{n,2})$.
\end{lemma}
The first statement is the classical Berry-Esseen theorem, see e.g.~\cite{dudley2018real}. As for the second statement, while there are many proofs of the multivariate Central Limit Theorem adapted to various situations, we provide here a short proof based on the arguments of~\cite{zhai2018multivariate,rio2009upper} at the end of this section.   

Now going back to the main argument, we start from Eq.~\eqref{eq:second} which we rewrite here: 
\begin{equation}
\E_{\pl}\big[\mu_{\bG}(E_{\delta,t}^c)\big] \le 
 \sum_{\by:d_{H}(\bx,\by) \le  n \delta} \P\Big(\big| \langle \bg_t , \by \rangle \big| \le \kappa \sqrt{n}  \,\Big|\, \big|\langle \bg , \bx \rangle \big| \le \kappa \sqrt{n} \Big)^m \, .\label{eq:second2}
\end{equation}
We will split the above sum based on the distance $d_{H}(\bx,\by)$. 
In the notation of Lemma~\ref{lem:clt} with $\xi = (x_iy_i)_{i=1}^n$, and letting $r= (t/n)\sum_{i=1}^n \xi_i$ we have
\begin{align}
\P\Big(\big| \langle \bg_t , \by \rangle \big| \le \kappa \sqrt{n} \,   \Big| \,  \big| \langle \bg , \bx \rangle \big| \le \kappa \sqrt{n} \Big) &=  \P\big(S_n \in [-\kappa,\kappa]^2\big) \big/  \P\big(S_{n,1} \in [-\kappa,\kappa]\big) \nonumber\\
&\le \frac{q_{\kappa}(r) + Cn^{-1/2}\log n}{p(\kappa) - Cn^{-1/2}} \, .\label{eq:estimateclt}
\end{align}
If $d_{H}(\bx,\by) \le K = n^{1/2-\epsilon_0}$, $\epsilon_0>0$ small enough, we use the bound $q_{\kappa}(r) \le q_{\kappa}(t) \le p(\kappa) - c\sqrt{1-t}$ as per Lemma~\ref{lem:estimates_q}. 
Restricting $t$ so that 
\[c\sqrt{1-t} \ge 2C \frac{\log n}{\sqrt{n}}\,,\] 
Eq.~\eqref{eq:estimateclt} raised to the power $m$ is bounded by 
\[\left(\frac{p(\kappa) - c\sqrt{1-t}/2}{p(\kappa) - Cn^{-1/2}}\right)^m \le e^{- c_0(\kappa) m \sqrt{1-t} + c_1(\kappa) m/\sqrt{n}} \le e^{- c' \sqrt{n} \log n}\,, \]
with $c' = c'(\alpha,\kappa)>0$. Therefore,
\begin{align}
\sum_{\by:d_{H}(\bx,\by) \le  K} \P\Big(\big| \langle \bg_t , \by \rangle \big| \le \kappa \sqrt{n}  \,\Big|\, \big|\langle \bg , \bx \rangle \big| \le \kappa \sqrt{n} \Big)^m 
&\le \sum_{k=0}^{\lfloor K \rfloor} \binom{n}{k} \cdot e^{-c' \sqrt{n} \log n} \nonumber\\
\le K \binom{n}{\lfloor K \rfloor} e^{-c' \sqrt{n} \log n}
&\le  e^{-c'\sqrt{n} (\log n)/2}\, . \label{eq:firstbd}
\end{align}  
Next for the terms where $ n^{1/2-\epsilon_0} < d_{H}(\bx,\by) \le K = \epsilon n$, we use the bound 
\[q_{\kappa}(r) \le q_{\kappa}(\langle \bx,\by\rangle/n) \le p(\kappa) - c\sqrt{d_{H}(\bx,\by)/n} \, .\]
Since $\sqrt{d_{H}(\bx,\by)/n} > 1/n^{1/4 + \epsilon_0} \gg (\log n)/\sqrt{n}$ for $\epsilon_0<1/4$, we similarly obtain 
\begin{align}
\sum_{\by:n^{1/2-\epsilon_0} < d_{H}(\bx,\by) \le \epsilon n} \P\Big(\big| \langle \bg_t , \by \rangle \big| \le \kappa \sqrt{n}  \,\Big|\, \big|\langle \bg , \bx \rangle \big| \le \kappa \sqrt{n} \Big)^m 
&\le n \binom{n}{\lfloor K \rfloor} \cdot e^{-c \sqrt{K n}} 
\le e^{-c' n}\, , \label{eq:secondbd}
\end{align}
for $\epsilon$ small enough and $c' = c'(\alpha,\kappa,\epsilon)>0$. 
Finally for $ \epsilon n < d_{H}(\bx,\by) \le \delta n$, we still use the bound $q_{\kappa}(r) \le q_{\kappa}(\langle \bx,\by\rangle/n)$ and compare the quantity in Eq.~\eqref{eq:estimateclt} to its Gaussian analogue $q_{\kappa}(\langle \bx,\by\rangle/n)/p(\kappa)$. By a Taylor expansion we have
\begin{align*}
  \frac{p(\kappa)}{q_{\kappa}(\langle \bx,\by\rangle/n)}  \P\Big(\big| \langle \bg_t , \by \rangle \big| \le \kappa \sqrt{n} \,   \Big| \,  \big| \langle \bg , \bx \rangle \big| \le \kappa \sqrt{n} \Big)
  &\le 1+ c (\log n)/\sqrt{n} \le e^{c n^{-1/2} \log n}\, ,
\end{align*}
where $c = c(\kappa)>0$. Therefore  
\begin{align}
\sum_{\by:\epsilon n < d_{H}(\bx,\by) \le \delta n} \P\Big(\big| \langle \bg_t , \by \rangle \big| \le \kappa \sqrt{n}  \,\Big|\, \big|\langle \bg , \bx \rangle \big| \le \kappa \sqrt{n} \Big)^m 
&\le e^{c m  n^{-1/2} \log n} \hspace{-0.5cm}\sum_{\by:\epsilon n < d_{H}(\bx,\by) \le \delta n}  \Big(\frac{q_{\kappa}(\langle \bx,\by\rangle/n)} {p(\kappa)}\Big)^m\,.\label{eq:thirdbd}
\end{align}
This last inequality puts as back in the setting of Gaussian disorder, at the price of the factor $e^{c m n^{-1/2} \log n}$ as above. The last term was shown to be exponentially small in~\cite{perkins2021frozen}:
\begin{lemma}[Consequence of Lemma 5 and Lemma 8 in \cite{perkins2021frozen}]\label{lem:bdexponential}
Under Gaussian disorder $\bG$, for all $\alpha< \alpha_{\sSAT}(\kappa)$, there exists $\epsilon< \delta $ and $c>0$ depending on $\alpha,\kappa$ such that 
\[ \E\Big[\Big|\Big\{\by \in S(\bG,\kappa)  \,:\, \epsilon n < d_{H}(\bx,\by) \le  n \delta\Big\}\Big| \, \big|\, \bx \in S(\bG, \kappa) \Big] = \hspace{-0.5cm}\sum_{\by:\epsilon n < d_{H}(\bx,\by) \le \delta n}  \Big(\frac{q_{\kappa}(\langle \bx,\by\rangle/n)} {p(\kappa)}\Big)^m \le e^{-c n}\, .
\]
\end{lemma}

Using the result of Lemma~\ref{lem:bdexponential} in Eq.~\eqref{eq:thirdbd} we obtain 
\begin{equation}\label{eq:fourthbd}
    \sum_{\by:\eps n < d_{H}(\bx,\by) \le \delta n} \P\Big(\big| \langle \bg_t , \by \rangle \big| \le \kappa \sqrt{n}  \,\Big|\, \big|\langle \bg , \bx \rangle \big| \le \kappa \sqrt{n} \Big)^m 
\le e^{-c' n} \, .
\end{equation}

Combining the bounds~\eqref{eq:firstbd}, \eqref{eq:secondbd} and \eqref{eq:fourthbd} we obtain 
\[\E_{\pl}\big[\mu_{\bG}(E_{\delta,t}^c)\big] \le e^{-c \sqrt{n} \log n} \, ,~~~ c>0\, ,\]
and this allows us to conclude once more. 

Finally, it remains to prove Lemma~\ref{lem:clt} to conclude the argument in the Rademacher case.

\begin{proof}[Proof of Lemma~\ref{lem:clt}]
We concentrate on proving the second bound in Eq.~\eqref{eq:clt}. For pedagogical reasons we first present a version of our argument which obtains the slower rate $1/n^{1/4}$, and then use a variation of this argument which achieves the rate $(\log n) / \sqrt{n}$. 

Let $G = (G_1,G_2)$ be a bivariate Gaussian random vector distributed as $(Z_0,Z_r)$, $r = (t/n)\sum_{i=1}^n\xi_i$. (Note that $\P(G \in [-\kappa,\kappa]^2) = q_\kappa(r)$.) For $\varepsilon>0$ and under any coupling of $(S_n,G)$ we have
\begin{align*}
\P\big(S_n \in [-\kappa,\kappa]^2\big) &\le \P\big(\|S_n\|_{\infty} \le \kappa, \|S_n-G\|_{\infty} \le \varepsilon \big) 
+ \P\big(\|S_n-G\|_{\infty} > \varepsilon \big) \\
&\le\P\big( \|G\|_{\infty} \le \kappa + \varepsilon \big) 
+ \P\big(\|S_n-G\|_{\infty} > \varepsilon \big) \, .
\end{align*}
On the one hand by Nazarov's inequality~\cite{nazarov2003maximal,chernozhukov2017detailed} (or direct calculations) we have 
\begin{equation*}
\P( \|G\|_{\infty} \le \kappa + \varepsilon) \le  \P(\|G\|_{\infty} \le \kappa) + C \varepsilon
\end{equation*}
for an absolute constant $C>0$. 
On the other hand by a union bound and Markov's inequality,
\begin{equation}\label{eq:tailbd}
\P\big(\|S_n-G\|_{\infty} > \varepsilon \big) \le \varepsilon^{-1} \E\big[\|S_n-G\|_1\big]\,. 
\end{equation}
Now consider coupling $S_n$ and $G$ coordinate-wise via increasing rearrangements: letting $F_n$ and $\Phi$ the cumulative distribution functions of the (real-valued) r.v.'s $S_{n,1}$ and $G_1$ respectively, we let $S_{n,j} = F_n(\Phi^{-1}(G_j))$, $j=1,2$. Then since the coordinates are identically distributed, 
\begin{equation*}
\E\big[\|S_n-G\|_1\big] =  W_1(S_{n,1},G_1) + W_1(S_{n,2},G_2) = 2  W_1(S_{n,1},G_1)\, ,
\end{equation*}
where $W_1(X,Y)$ is the Wasserstein-1 distance between $X$ and $Y$, i.e., the infimum of $\E[|X-Y|]$ over all couplings with the correct marginals. This infimum is achieved by the increasing rearrangement mentioned above.   
A CLT in the $W_1$ distance, equivalently a Berry-Esseen bound for the $L_1$ distance (rather than $L_{\infty}$) between the distribution functions is classical: 
\begin{equation*}
W_1(S_{n,1},G_1) \le \frac{C}{\sqrt{n}} \, ,
\end{equation*}
for some absolute constant $C>0$, see~\cite{esseen1958mean,goldstein2010bounds}.  Combining these bounds we obtain 
\begin{equation}\label{eq:1stbd}
\P\big(S_n \in [-\kappa,\kappa]^2\big) \le \P\big(G \in [-\kappa,\kappa]^2\big) + C \varepsilon + \frac{2C}{\varepsilon\sqrt{n}} \,.
\end{equation}
We obtain the bound $1/n^{1/4}$ by taking $\varepsilon = 1/n^{1/4}$. 

Let us now improve the convergence rate $1/n^{1/4}$ obtained above to $(\log n) / \sqrt{n}$. We  use a stronger result of~\cite{rio2009upper} on a Wasserstein CLT in $\psi$-Orliz norm with $\psi(x) = e^{|x|}-1$: Define the $\psi$-Orliz norm  of a real-valued random variable $X$ as 
\[\|X\|_{\psi} = \inf\{a > 0 \,:\, \E\psi(X/a) \le 1\} \, .\]
This is also known as the \emph{sub-exponential norm} of the random variable $X$~\cite{vershynin2018high}.
Define the Wasserstein-$\psi$ distance between two r.v.'s $X$ and $Y$ as
\[ W_{\psi}(X,Y) = \inf \,\|X-Y\|_{\psi} \, ,\]
where the infimum is taken over all couplings of $X$ and $Y$.  
Now we follow the same argument as before and replace the bound~\eqref{eq:tailbd} by the following improved bound
\begin{align*}
\P\big(\|S_n-G\|_{\infty} > \varepsilon \big) &\le \P\big(|S_{n,1}-G_1| > \varepsilon \big) + \P\big(|S_{n,2}-G_2| > \varepsilon \big) \\
&\le 2 e^{-\varepsilon/a} \E \big[e^{|S_{n,1}-G_1|/a}\big] \, ,
\end{align*}
valid for any $a>0$. We take $a = \|S_{n,1}-G_1\|_{\psi}$ so that $\E \big[e^{|S_{n,1}-G_1|/a}\big] -1\le 1$ by definition of the $\psi$-Orliz norm and obtain 
\[\P\big(\|S_n-G\|_{\infty} > \varepsilon \big) \le 4 e^{-\varepsilon/\|S_{n,1}-G_1\|_{\psi}}\,.\]
We now choose the optimal coupling achieving the $W_\psi$ distance (this turns out to be the increasing rearrangement mentioned above), and use the univariate CLT for $W_\psi$ in~\cite[Theorem 2.1]{rio2009upper}: 
\begin{equation*}
W_\psi(S_{n,1},G_1) \le \frac{C}{\sqrt{n}} \, ,
\end{equation*}
to obtain the bound 
\begin{equation}\label{eq:2ndbd}
\P\big(S_n \in [-\kappa,\kappa]^2\big) \le \P\big(G \in [-\kappa,\kappa]^2\big) + C \varepsilon +  4 e^{-\varepsilon \sqrt{n}/C} \,,
\end{equation}
to be compared with~\eqref{eq:1stbd}. Now we can take $\varepsilon = C (\log n) / \sqrt{n}$ to conclude. 
\end{proof}

\section{Consequences for stable algorithms and circuits}\label{section:Stable-circuits}
Let  $\ALG_n :    \R^{m \times n} \times \Omega \to \{-1,+1\}^n$ be a family of randomized algorithms taking the matrix $\bG \in \R^{m \times n}$ with $m = \lfloor\alpha n\rfloor$ and an independent random seed $\omega \in \Omega$ as input. The margin $\kappa>0$ is also an input to the algorithm, but we omit it from our notation since it will be considered fixed throughout. 

\begin{definition}\label{def:stable}
 Let $\mu^{\salg}_{\bG}$ be the law of the output of the algorithm $\bx^{\salg} = \ALG_n(\bG , \omega) \in \{-1,+1\}^n$, conditional on $\bG$: $\mu^{\salg}_{\bG} := \Law(\bx^{\salg} \,|\,\bG)$.  
For a sequence $t_n \to 0$, we say that  $(\ALG_n)_{n \ge 1}$ is $t_n$-stable if 
\begin{equation}
\lim_{n \to \infty} \, \frac{1}{\sqrt{n}} W_{2}\big(\mu^{\salg}_{\bG} , \mu^{\salg}_{\bG_{t_n}}\big) = 0 \,  ~~~\mbox{in prob.}
\end{equation}
\end{definition}
A straightforward consequence of this definition, as observed in~\cite{alaoui2022sampling}, is that $t_n$-stable algorithms cannot approximately sample from a distribution exhibiting transport disorder chaos with the same parameter $t_n$, in the following sense: 
\begin{lemma}\label{lem:impossibility}
If the pairs $(\mu_{\bG},\mu_{\bG_{t_n}})_{n \ge 1}$ satisfy the conclusion of Theorem~\ref{thm:w_2-chaos} and $(\ALG_n)_{n \ge 1}$ is $t_n$-stable, then
\begin{equation}
\E W_2( \mu_{\bG}, \mu^{\salg}_{\bG}) \ge \sqrt{\delta n}\, ,   
\end{equation}
for some $\delta>0$.
\end{lemma}
\begin{proof}
This follows from the triangle inequality
\[
W_2(\mu_{\bG}, \mu_{\bG_t}) \le W_2( \mu_{\bG}, \mu^{\salg}_{\bG}) + W_2( \mu^{\salg}_{\bG}, \mu^{\salg}_{\bG_t}) + W_2(\mu^{\salg}_{\bG_t}, \mu_{\bG_t})\,,
\]
for any $t$, and the fact that the first and third terms on the right-hand side have the same distribution.   
\end{proof}

Many known algorithmic schemes are stable as per the above definition. Proving that a sampling algorithm $\ALG_n$ is $t_n$-stable for some $t_n$ usually reduces to showing that the map $\bG \mapsto \ALG_n(\bG,\omega)$ is suitably Lipschitz with a constant independent of $n$. In the case of iterative algorithms the Lipschitz constant will usually depend on the number of iterations. As already mentioned in the introduction, it was shown, e.g., in~\cite[Proposition 5.2]{alaoui2022sampling} that the class of iterative algorithms using Lipschitz non-linearities is $t_n$-stable for some $t_n \to 0$ when ran for a constant number of iterations.

Another class of approximate sampling algorithms we are especially interested in this paper are Boolean circuits of small depth and small size. We will show that circuits are $t_n$-stable for some $t_n \to 0$ depending on the depth and size.  

\paragraph{Boolean circuits}
We begin by recalling some properties of Boolean circuits with controlled sizes and depth. Let $\Sigma_r=\{-1,+1\}^r$. Fix any $N,n$ and consider
a Boolean circuit $\cC:\Sigma_N\to \Sigma_n$. We write $C_1,\ldots,C_n: \Sigma_N\to \{-1,+1\}$ for the coordinates of the circuit associated with each of the $n$
output nodes. 
Let $s(\cC)$ denote the size of the circuit, namely the number of gates and let $d(\cC)$ denote the depth of the circuit which is the length of the longest path from the input  layer to the output layer.

Let $m,n$ be as in our setting of symmetric perceptron model with Rademacher disorder. Fix $R$ with $\omega=(\omega_i, i\in [R])\in \Sigma_R$ representing the random seeds. 
We consider a Boolean circuit $\cC: \Sigma_{mn+R}\to \Sigma_n$ where the first $mn$ entries are the entries of the matrix $\bG \in \{-1,+1\}^{m \times n}$ representing the instance of the perceptron model, and the remaining entries $\omega = (\omega_i, i\in [R])$ are the random seeds i.i.d.\ Rademacher, independently generated from $\bG$. The total input is described as a pair $(\bG,\omega)$.

\begin{theorem}\label{theorem:stability-boolean-circuits}
Fix $\alpha>0$, $m= \lfloor\alpha n\rfloor$ and $R>0$ and let $\ALG_n = \cC : \Sigma_{mn+R} \to \Sigma_n $ be a Boolean circuit. 
For every $\gamma\in (0,1)$ there exist $\beta,A>0$ such that 
$(\ALG_n)_{n \ge 1}$ is $t_n$-stable 
as per Definition~\ref{def:stable} for any $t_n \ge 1-n^{-1+\gamma}$, where both $\bG$ and $\bG_t$ are Rademacher with entrywise correlation equal to $t_n$, provided that at least
 one of the following conditions on 
the depth and size of the circuit holds:
\begin{enumerate}
\item Polynomial size and (sub)logarithmic depth: $s(\cC) \le n^A$,  $d(\cC) \le \beta \log n / \log \log n$,
\item Sub-exponential size and bounded depth: $s(\cC) \le \exp(n^{\beta})$, $d(\cC) \le A$. 
\end{enumerate}
\end{theorem}

We note that the dependence on $R$ in the above statements is  implicit since $mn + R$ is a trivial lower bound on the circuit size $s({\cC})$. 
An immediate corollary is the following:
\begin{corollary}
Boolean circuits under either of the above depth and size constraints cannot approximately sample from $\mu_{\bG}$, in the sense of Lemma~\ref{lem:impossibility}
\end{corollary}
\begin{proof}
Indeed, Boolean circuits are $t_n$-stable if $\sqrt{n(1-t_n)} \le n^{\gamma/2}$ as per Theorem~\ref{theorem:stability-boolean-circuits}, and we recall that the pair $(\mu_{\bG},\mu_{\bG_{t_n}})$ satisfies the conclusion of Theorem~\ref{thm:w_2-chaos}, i.e., has transport disorder chaos, with Rademacher disorder if $\sqrt{n(1-t_n)} \ge C \log n$. The two conditions are compatible.    
\end{proof}

Theorem~\ref{theorem:stability-boolean-circuits} is a consequence of a more precise stability bound for Boolean circuits given below.

\begin{theorem}\label{theorem:Boolean-circuit-insensitivity}
For every $c,\beta,\gamma>0$ satisfying  $0<\beta+\gamma<c<1$ the following holds for all  $A>0$ and all sufficiently large $n$. 
Let $\cC : \Sigma_{mn+R} \to \Sigma_n $ be a Boolean circuit with $d(\cC)\le \beta \log n/\log\log n$ and $s(\cC)\le n^A$. Suppose further that  
 \begin{align}\label{eq:gamma}
t\ge 1-n^{-1+\gamma}.
 \end{align}
 Let $X=(\bG,\omega),X_t=(\bG_t,\omega)$.
 Then  
\begin{align}\label{eq:C-change}
\E{ \|{\cC}(X_t) - {\cC}(X)\|_2^2} \le n^{c}.
\end{align}
Similarly, fix  $c,A>0$ and $\gamma<c<1$. Suppose  $0<\beta A+\gamma<c$. Then~\eqref{eq:C-change} 
holds  for all sufficiently large $n$, provided 
 $s({\cC})\le \exp(n^{\beta})$, $d({\cC})\le A$ and $t$ satisfying~\eqref{eq:gamma}.
\end{theorem}

The way to read this claim in a ``friendly" way is as follows. Suppose we want to guarantee algorithmic
$t_n$-stability with any fixed stability parameter $\gamma\in (0,1)$ (so that $t_n=1-n^{-1+\gamma})$.
We then fix any $c<1$ such that $\gamma<c$, and find $\beta$ small enough so that $\beta<c-\gamma$. 
Then per theorem above, poly-size circuits with depth at most $\beta\log n/\log\log n$ are $t_n$-stable, since $\E [W_2(\mu^{\salg}_{\bG},\mu^{\salg}_{\bG_t})^2]$ is upper-bounded by the left-hand side of Eq.~\eqref{eq:C-change} 

Similarly, fixing $A>0$ arbitrarily and letting $\beta$ be small enough so that $\beta<(c-\gamma)/A$, we obtain
that the circuit is $t_n$-stable (again with parameter $\gamma$), provided the circuit has depth at most $A$ and size at most $\exp(n^\beta)$. Therefore Theorem~\ref{theorem:Boolean-circuit-insensitivity} implies Theorem~\ref{theorem:stability-boolean-circuits}.

\begin{proof}[Proof of Theorem~\ref{theorem:Boolean-circuit-insensitivity}]

Let us briefly recall the following notions from Boolean analysis/Fourier analysis on $\{-1,+1\}^N$; see e.g.~\cite{o2014analysis} for a reference.
Consider the standard Fourier expansion of functions on $\{-1,+1\}^{N}$ associated with the
uniform measure on $\{-1,+1\}^N$. The basis of this expansion are monomials of the form
$x_S\triangleq \prod_{i\in S}x_i,\, S\subset [N]$. For every function $g:\{-1,+1\}^{N}\to \R$, the associated
Fourier coefficients are 
\begin{align*}
\hat g_S=\E[g(X) X_S], \qquad S\subseteq [N],
\end{align*}
where $X$  a uniformly random vector in $\{-1,+1\}^N$. 
Then the Fourier expansion of $g$ is 
\begin{equation}\label{eq:fourier}
g(x)=\sum_{S\subseteq [N]} \hat g_S x_S\,,
\end{equation}
and the Parseval (or Walsh) identity states that 
\begin{equation}\label{eq:parseval}
\E[g(X)^2] = \sum_{S \subseteq [N]} \hat{g}_S^2\,.
\end{equation}

Let $X_t \in \{-1,+1\}^N$ be a random vector obtained by independently resampling every coordinate of $X$ with probability $1-t$, so that for each $i \in [N]$, $X_i$ and $(X_{t})_i$ have correlation $t$. Then applying~\eqref{eq:fourier} yields
\begin{equation}\label{eq:corr}
\E \big[g(X)g(X_t)\big] = \sum_{S\subseteq [N]} t^{|S|} \hat{g}_S^2 \,.
\end{equation}

We now consider the functions $C_j$  (the $j$th component of ${\cC}$, $1 \le j \le n$) as functions of the first $N = mn$ coordinates of $X$ corresponding to the matrix $\bG$; the random seed $\omega$ being fixed throughout the analysis.
We fix  $\delta>0$ and let $\hat C_{S,j}, S\subseteq [N]$ be the Fourier coefficients of the function $C_j$. Fix a constant $c_0>0$ and let
\begin{align}
D &=\log(1/\delta)\left(c_0\log {s({\cC})\over \delta}\right)^{d({\cC})-1} \label{eq:bound-Dn}. 
\end{align}
The Linial-Mansour-Nisan Theorem~\cite[page 106]{o2014analysis}
 states that for some choice of the universal constant $c_0$ the spectrum of a depth-$d$ circuit with size $s$
is $\delta$-concentrated on degree $\le D$. More precisely, in our context, it states that for each $j=1,2,\ldots,n$,
\begin{align*}
\sum_{S\subseteq [N] \,:\, |S|> D}(\hat C_{S,j})^2\le \delta\, .
\end{align*} 

Since $X$ and $X_t$ are equally distributed, using~\eqref{eq:parseval} and~\eqref{eq:corr} we have 
\begin{align}
\E[ (C_j(X)-C_j(X_t))^2] &=  2\big(\E[ C_j(X)^2] - \E[ C_j(X)C_j(X_t)]\big)\notag\\
&= 2\sum_{S\subseteq [N]} (1-t^{|S|}) (\hat C_{S,j})^2\notag\\
&\le 2 (1-t^D)\sum_{S \,:\, |S|\le D}  (\hat C_{S,j})^2 + 2 \delta \notag\\
&\le  2 (1-t^D) + 2\delta\,.\label{eq:Delta}
\end{align}
where we applied the LMN theorem as stated above in the second-to-last last line and appropriately bounded the term $1-t^{|S|}$ depending on whether $|S|$ is larger or smaller than $D$. The last line follows by the fact $\sum_{S \subseteq [N]} \hat C_{S,j}^2 = \E{C_j(X)^2} =1$ since $C_j(X) \in \{-1,+1\}$.

Summing over $j\in [n]$ and using~\eqref{eq:Delta} we obtain
\begin{align*}
\E{ \|C(X)-C(X_t)\|_2^2 }\le 2n(1-t^D)+2n\delta \, .
\end{align*}
We now set $\delta=n^{c-1}$, $c<1$. 
The second term in the above right-hand side is bounded by $2n^c$, and from Eq.~\eqref{eq:bound-Dn}  we obtain
\begin{align*}
D &\le (1-c)(\log n)\left(c_0(A+1-c)\log n\right)^{d({\cC})-1}\\
&\le c'  \left(c_0(A+1-c)\log n\right)^{d({\cC})}\,,
\end{align*}
for some constant $c'>0$ depending on $A, c_0,c$. 
We further have
\begin{align*}
\left(\log n\right)^{d({\cC})} \le \left(\log n\right)^{\beta \log n/\log\log n}= n^{\beta},
\end{align*}
implying $D\le n^{\beta(1+o(1))}$.
Thus recalling our bound $1-t\le 1/n^{1-\gamma}$  we obtain
\begin{align*}
n(1-t^D)\le n(1-t) D = n^{\beta(1+o(1))}n^{\gamma}\, .
\end{align*}
Since $\beta+\gamma<c$ we have $Dn(1-t)=o(n^{c})$ and the first claimed bound $n^{c}$ is obtained for $n$ large enough.

To prove the second part recall bounds $s({\cC}) \le e^{n^\beta}$ and $d({\cC}) \le A$.    
We again let $\delta=n^{c-1}$. 
Then from  (\ref{eq:bound-Dn}) we obtain
\begin{align*}
D&\le
(1-c)(\log n)\left(c_0\log \left(s({\cC})n^{1-c}\right) \right)^{A-1} \\
&\le 
(1-c)(\log n)\left(c_0n^\beta+c_0(1-c)\log n \right)^{A-1}\, .
\end{align*}
Thus $D \le n^{\beta A(1+o(1))}$, implying
$n(1-t)D \le n^{(\beta A+\gamma)(1+o(1))}$, again due to our bound on $t$.
Since $\beta A+\gamma<c$ we obtain $Dn(1-t)=o(n^{c})$. This completes the proof of the theorem.
\end{proof}

\vspace{5mm}
\textbf{Acknowledgments.} The first author would like to thank Will Perkins for stimulating discussions in the early stages of this project. We thank the anonymous referees for their feedback and for suggesting a shorter proof of Theorem~\ref{theorem:Boolean-circuit-insensitivity}.     
The first author acknowledges the support of NSF grant DMS  2450867. 
The second author acknowledges the support of NSF grant CISE  2233897. 
\vspace{5mm}

 \newpage
\bibliographystyle{amsalpha}
 \bibliography{bibliography}

\newpage
\appendix

\section{The derivative of $q_{\kappa}(t)$}
\label{app:blah}
Here we show that
\[q_{\kappa}'(t) = \frac{1}{\pi} \frac{e^{-\kappa^2/2}}{\sqrt{1-t^2}} \left(e^{-\kappa^2(1-t)/2(1+t)} -e^{-\kappa^2(1+t)/2(1-t)}\right) \, .\]
Indeed, let $\phi$ denote standard normal density, $Z, Z' \sim N(0,1)$ i.i.d., and define for $0 \le t < 1$,
\[a_{\pm}(t) = \frac{\pm \kappa - tZ}{\sqrt{1-t^2}}\,.\]
We first have
\[a_{\pm}'(t) =  - \frac{ Z}{\sqrt{1-t^2}} +  \frac{(\pm \kappa  - tZ)t}{(1-t^2)^{3/2}}
=\frac{\pm \kappa t - Z}{(1-t^2)^{3/2}}\,.\]
Next,
\begin{align*}
q_{\kappa}'(t) &= \frac{\rmd }{\rmd t} \P\left( |Z| \le \kappa \, , |t Z + \sqrt{1-t^2} Z'| \le \kappa\right) \\
&= \frac{\rmd }{\rmd t} \E\Big[ \indi\{|Z| \le \kappa\} \int_{a_{-}(t)}^{a_{+}(t)} \phi(y) \rmd y \Big] \\
&= \E\Big[ \indi\{|Z| \le \kappa\} \big(a_{+}'(t) \phi(a_{+}(t)) - a_{-}'(t) \phi(a_{-}(t))\big)\Big] \\
&= \frac{1}{\sqrt{2\pi}(1-t^2)^{3/2}}\E\Big[ \indi\{|Z| \le \kappa\} \Big((\kappa t - Z)e^{-(\kappa - tZ)^2/2(1-t^2)} \\
&\hspace{5.7cm} + (\kappa t + Z) e^{-(\kappa + tZ)^2/2(1-t^2)} \Big)\Big]\, .
\end{align*}
Next we observe that the two terms inside the expectation are images of each other by a sign flip of $Z$, so 
\begin{align*}
q_{\kappa}'(t) 
&= \frac{2}{\sqrt{2\pi}(1-t^2)^{3/2}}\E\Big[ \indi\{|Z| \le \kappa\} (\kappa t - Z)e^{-(\kappa - tZ)^2/2(1-t^2)} \Big]\\
&= \frac{2}{\sqrt{2\pi}(1-t^2)^{3/2}} \int_{-\kappa}^{\kappa} (\kappa t - z) e^{-(\kappa - tz)^2/2(1-t^2) -z^2/2} \frac{\rmd z}{\sqrt{2\pi}}\\
&= \frac{2e^{-\kappa^2/2}}{\sqrt{2\pi}(1-t^2)^{3/2}}  \int_{-\kappa}^{\kappa} (\kappa t - z) e^{-(\kappa t- z)^2/2(1-t^2)} \frac{\rmd z}{\sqrt{2\pi}}\, ,
\end{align*}
where the last line follows from the identity $(\kappa  -tz)^2 + (1-t^2)z^2 = \kappa^2(1-t^2) + (\kappa t -z)^2$. 
By the change of variables $z' = (\kappa t - z)/\sqrt{1-t^2}$ we obtain
\begin{align*}
q_{\kappa}'(t) 
&=\frac{2e^{-\kappa^2/2}}{\sqrt{2\pi(1-t^2)}} \int_{-\kappa(1-t)/\sqrt{1-t^2}}^{\kappa(1+t)/\sqrt{1-t^2}} z' e^{-z'^2/2} \frac{\rmd z'}{\sqrt{2\pi}}\\
&=\frac{2e^{-\kappa^2/2}}{\sqrt{2\pi(1-t^2)}} \Big(\phi\Big(\kappa \sqrt{\frac{1-t}{1+t}}\Big)-\phi\Big(\kappa \sqrt{\frac{1+t}{1-t}}\Big)\Big)\, .
\end{align*}

\end{document}